\documentclass[12pt]{amsart}
\usepackage[utf8]{inputenc}
\usepackage[foot]{amsaddr}
\usepackage{amsmath,amsfonts,amsthm,amssymb,verbatim,etoolbox,color}
\usepackage{enumitem} 
\usepackage{epigraph}
\usepackage{appendix}
\usepackage{xfrac}
\usepackage{mathtools}
\usepackage{flexisym}
\usepackage{bbm}
\usepackage{float}
\usepackage{graphicx}

\usepackage[margin=1in]{geometry}
\usepackage[
    backend=biber,
    maxnames=4,
    maxalphanames=4,
    style=alphabetic,
    backref
  ]{biblatex}
\usepackage{scalerel}
\newcommand*{\TextVCenter}[1]{%
  \text{$\vcenter{\hbox{#1}}$}%
}
\DeclareMathOperator*{\E}{\TextVCenter{\scaleobj{2}{\mathbb{E}}}}

\usepackage{hyperref}
\hypersetup{
    colorlinks=true,
    linkcolor=red,
    filecolor=magenta,      
    urlcolor=magenta,
    pdftitle={Sumsets of Ahlfors--David regular sets},
    pdfauthor = {Fred Tyrrell},
    pdfpagemode=FullScreen,
    linktocpage = true,
    citecolor = blue,
    }

\usepackage{theoremref}

\newtheorem{thm}{Theorem}
\newtheorem{lemma}[thm]{Lemma}

\newtheorem{prop}[thm]{Proposition}
\theoremstyle{definition}
\newtheorem{defn}[thm]{Definition}

\numberwithin{thm}{section}

\newcommand{\N}{\mathbb{N}}

\newcommand{\Z}{\mathbb{Z}}
\newcommand{\R}{\mathbb{R}}
\renewcommand{\P}{\mathbb{P}}

\newcommand{\EE}{\mathcal{E}}

\newcommand{\PP}{\mathcal{P}}

\newcommand{\sub}{\subseteq}

\newcommand{\bra}[1]{\left(#1\right)}
\newcommand{\sqbra}[1]{\left[#1\right]}

\newcommand{\ceil}[1]{\left\lceil #1 \right\rceil}
\newcommand{\floor}[1]{\left\lfloor #1 \right\rfloor}
\newcommand{\h}[1]{\widehat{#1}}

\newcommand{\abs}[1]{\left\lvert {#1} \right \rvert}
\DeclareMathOperator\supp{supp}

\newcommand{\al}{\alpha}

\newcommand{\be}{\beta}
\newcommand{\de}{\delta}

\def\bpm{\begin{pmatrix}}
\def\epm{\end{pmatrix}}

\title{Sumsets of Ahlfors--David regular sets}
\author{Fred Tyrrell}
\date{\today}
\email{fred.tyrrell@bristol.ac.uk}
\address{Fry Building, School of Mathematics, University of Bristol}
\urladdr{http://fredtyrrell.com}

\begin{document}

\begin{abstract}
We prove that there is an absolute constant $c>0$ such that for every finite $(\alpha,M)$-Ahlfors--David regular set $A\subseteq[N]$ with $0\leq\al<1$ and $M\geq2$, we have the sumset estimate
\[\abs{A+A}\geq \abs{A}^{1+c\bra{1-\al}/\log M}.\]
We also prove a corresponding statement for Ahlfors--David regular sets in $[0,1]$. The proof combines a multiscale entropy decomposition with inverse results from additive combinatorics, showing that Ahlfors--David regularity forces a definite entropy gain at each scale. The dependence $1/\log M$ in the exponent is of optimal order. 
\end{abstract}

\maketitle

\section{Introduction}

\subsection{The discrete setting}
\begin{defn}\thlabel{ADregular}
For a non-empty finite set $A \sub [N],\footnote{We use the standard shorthand notation $[N]=\{1,\ldots,N\}$.}$ we say $A$ is \emph{Ahlfors--David regular} with dimension $\al$ and regularity constant $M$ if, for every $a \in A$ and every $1 \leq R \leq N$,
\[\frac{1}{M}R^{\al}\leq \abs{A \cap [a-R,a+R]} \leq M R^{\al}.\]
\end{defn}
Ahlfors--David regularity prevents a set from concentrating too strongly at any location or scale, and it is natural to ask how much additive expansion this forces. Our main result gives the optimal order of dependence on the regularity constant.

\begin{thm}\thlabel{main}
There exists a constant $c>0$ such that for every $N \geq 1$, and every Ahlfors--David regular set $A \sub [N]$ of dimension $0 \leq \al < 1$ and regularity constant $M \geq 2$,
\[\abs{A+A}\geq \abs{A}^{1+\frac{c\bra{1-\al}}{\log M}}.\]
\end{thm}

\subsection{The continuous setting}
\thref{main} also gives a corresponding statement for the usual notion of Ahlfors--David regularity, in terms of the covering number of a set. For a bounded set $E\sub\R$, let $N_\delta(E)$ denote the least number of intervals of length $\delta$ required to cover $E$. We also write $B\bra{x,r}=[x-r,x+r]$ as shorthand for the interval of radius $r$ centred at $x$.

\begin{defn}
Let $0 < \de \leq 1$. A non-empty compact set $X \sub [0,1]$ is said to be $\al$-Ahlfors--David regular on scales $[\delta,1]$ with regularity constant $C$ if there is a finite Borel measure $\mu$ supported on $X$ such that
\[C^{-1}r^\al \leq \mu(B(x,r)) \leq C r^\al\]
for every $\delta \leq r \leq 1$, whenever $x \in X$.
\end{defn}

The following is the continuous version of \thref{main}.

\begin{thm}\thlabel{continuous}
There is a constant $c > 0$ such that, for any $\al$-Ahlfors--David regular set $X\sub[0,1]$ on scales $[\delta,1]$ with regularity constant $C \geq 2$ and dimension $0 < \al < 1$,
\[N_\delta(X+X)\geq \frac{1}{3}N_\delta(X)^{1+\frac{c\bra{1-\al}}{\log C}}.\]
\end{thm}

We record the details of this standard reduction in Section 6.

\subsection{Background and motivation}
The additive structure of Ahlfors--David regular sets has previously been studied via additive energy. Dyatlov and Zahl \cite{dyatlov2016spectral}, motivated by applications to the fractal uncertainty principle, proved a power-saving additive-energy estimate for regular subsets of the line. Cladek and Tao \cite{cladek2021additive} subsequently obtained substantially stronger quantitative estimates, including polynomial dependence on the regularity constant in one dimension. By Cauchy--Schwarz, these energy estimates imply corresponding lower bounds for the size of the sumset, but their dependence on the regularity constant remains far from the logarithmic scale. A digit construction of Dyatlov and Zahl \cite[Section 6.8.3]{dyatlov2016spectral}, and its finite-scale discretisations, shows that the exponent gain in a sumset estimate cannot in general be larger than order $1/\log M$. Thus \thref{main} attains the optimal order of additive expansion for Ahlfors--David regular sets. We give the details of this construction in the appendix.

\smallbreak
More broadly, additive growth of fractal sets and measures has been studied from a number of perspectives, including dimension growth for sumsets and convolutions, projection theorems, and discretised sum-product phenomena. This includes work of Fraser, Howroyd and Yu \cite{fraser2019dimension}, Rossi and Shmerkin \cite{rossi2020measures}, Orponen \cite{orponen2022arithmetic}, Máthé and O'Regan \cite{https://doi.org/10.1112/jlms.70389}, and O'Regan \cite{o2025sum}.
\smallbreak
In forthcoming work, Murphy \cite{murphy} obtains a near-optimal additive-energy estimate which, by Cauchy--Schwarz, implies 
\[\abs{A+A}\geq \abs{A}^{1+\frac{c_\al}{\log M \log \log \log M}}.\] 
To iterate across scales, Murphy prunes the set to obtain approximately uniform branching. The pruning means that the local doubling depends on $M$, and the quantitative version of Freiman's theorem due to Raghavan \cite{raghavan2025improved} produces the additional $\log \log \log M$ factor. The present argument avoids the need for pruning, since the chain rule for conditional entropy allows us to retain the original uniform distribution on $A$ and accumulate fixed local entropy gains directly across scales. Combined with the recent entropic inverse theorem of Green, Manners and Tao \cite{green2025sumsets}, this allows Freiman's theorem to be applied with a fixed doubling constant, yielding the optimal $1/\log M$ dependence.

\subsection{Proof sketch}
Let $X$ and $Y$ be independent uniform random variables on $A$. Then
\[d[X,-Y]=H(X+Y)-H(X)\leq\log\frac{\abs{A+A}}{\abs A},\]
where $d$ denotes entropic Ruzsa distance and $H$ is Shannon entropy. Thus it is enough to show that $X+Y$ has substantially more entropy than $X$.
\smallbreak
We decompose $A$ across a sequence of scales whose ratio is a sufficiently large fixed power of $M$. Ahlfors--David regularity ensures that, for most locations and scales, the resulting local distributions are spread over many subintervals and have no unusually large probabilities on a single value. If two such distributions have small entropic Ruzsa distance, an inverse theorem of Green, Manners and Tao \cite{green2025sumsets}, followed by Freiman's theorem, would force a positive proportion of one of the distributions to lie in a generalised arithmetic progression. A one-dimensional fibre of this progression would then give a long arithmetic progression containing too many occupied subintervals, contradicting the sparsity forced by Ahlfors--David regularity.
\smallbreak
Finally, the chain rule for conditional entropy allows these fixed local gains to be accumulated across scales, with a loss of at most $\log 2$ at each step from a possible carry. Since consecutive scales differ by a fixed power of $M$, for a fixed $\al$ there are on the order of $\log N/\log M$ scales. Tracking the dependence on $\al$ gives
\[d[X,-Y]\gg\al\bra{1-\al}\frac{\log N}{\log M}.\]
Regularity gives $\abs A\leq MN^\al$, and hence $\log\abs A\leq2\al\log N$. Consequently,
\[\log\frac{\abs{A+A}}{\abs A}\gg\bra{1-\al}\frac{\log\abs A}{\log M},\]
with an absolute implicit constant.

\subsection{Higher dimensions}
\thref{main,continuous} both extend to higher dimensions. Fix $D\geq1$ and a norm $\|\cdot\|$ on $\R^D$, and write 
\[B(x,R)=\{y\in\R^D:\|y-x\|\leq R\}.\]
We say that a non-empty set $A\sub[N]^D$ is $(\al,M)$-Ahlfors--David regular if 
\[M^{-1}R^\al\leq\abs{A\cap B(a,R)}\leq MR^\al\] 
for every $a\in A$ and $1\leq R\leq N$. For every $0\leq\al<1$, the argument of this paper gives 
\[\abs{A+A}\geq\abs A^{1+\frac{c_{\|\cdot \|,D}\bra{1-\al}}{\log M}}\] 
for some $c_{\|\cdot \|,D}>0$, where the constant may depend on the chosen norm and $D$.
\smallbreak
Similarly, if $0<\al<1$ and a compact set $X\sub[0,1]^D$ supports a finite Borel measure $\mu$ satisfying 
\[C^{-1}r^\al\leq\mu(B(x,r))\leq Cr^\al\] 
for every $x\in X$ and $\de\leq r\leq1$, then 
\[N_\de(X+X)\gg_{D,\|\cdot \|} N_\de(X)^{1+\frac{c_{\|\cdot \|,D}\bra{1-\al}}{\log C}},\] 
where $N_\de$ may be defined using balls, or equivalently cubes, of scale $\de$.
\smallbreak
To obtain these extensions, we first work with the supremum norm and replace the intervals in the multiscale decomposition by coordinate cubes. In the packing argument of \thref{good}, we partition the cubes into $2^D$ classes according to the parity of their coordinate indices. This gives the same conclusions with constants depending on $D$, since after adjusting the threshold defining a good cube, the probability of lying in a good cube is still at least $9/10$.
\smallbreak
The arithmetic progression avoidance argument also extends to $\Z^D$. For a progression with step $t$, choosing a coordinate for which $|t_i|=\|t\|_\infty$ gives the required bounded-overlap estimate. The entropy inverse theorem and Freiman's theorem apply in $\Z^D$, so the local entropy growth argument proceeds as before.
\smallbreak
In the entropy decomposition, the carry is now a vector in $\{0,1\}^D$, giving a loss of at most $D\log 2$ at each scale, instead of $\log 2$. Choosing the local entropy threshold sufficiently large in terms of $D$ allows the gains to be accumulated as in Section 5. Tracking the constants gives a scale exponent $C=\ceil{K_D/(\al(1-\al))}$, where $K_D$ depends only on $D$. In the supremum norm, regularity at radius $N$ gives $\abs A\leq MN^\al$, so the final comparisons apply with constants depending only
on $D$. Other fixed norms can then be handled by norm equivalence and bounded covering arguments, with constants depending additionally on the chosen norm.
\smallbreak
The restriction $\al<1$ comes from the arithmetic progression avoidance part of the proof. The same obstruction appears in the work of Dyatlov and Zahl \cite[Section 6.8.2]{dyatlov2016spectral}, who note that their higher-dimensional argument also breaks down at $\al = 1$ because Ahlfors--David regularity alone no longer prevents long arithmetic progressions. Cladek and Tao \cite{cladek2021additive} obtain additive expansion more generally for non-integral $0 < \al < D$, using different geometric input.
\smallbreak
Since there are no new ideas needed to extend \thref{main,continuous} to higher dimensions, we have chosen to prove only the one dimensional versions of our results, to avoid unnecessary parameters and streamline the arguments.

\section*{Acknowledgements}
I would like to thank Misha Rudnev for suggesting the problem, useful discussions, and helpful feedback on earlier drafts of the paper. I would also like to thank Brendan Murphy for sharing an early draft of his paper.
\smallbreak
The author is supported by the Heilbronn Institute for Mathematical Research.

\section{Preliminaries}
We collect the relevant definitions, identities and inequalities related to entropy. We omit the proofs, since all of these results are standard. For additional background and proofs, see for example \cite{tao2010sumset,ruzsa2009sumsets}.
\smallbreak
The (Shannon) entropy for a finitely supported random variable $X$ is defined as
\[H(X) = \sum_{x}\P\bra{X=x} \log\frac{1}{\P\bra{X=x}}.\footnote{By convention, we define $0 \log \frac{1}{0} = 0$.}\]
For two finitely supported random variables $X$ and $Y$, the conditional entropy is defined by
\[H\bra{X|Y}=\sum_{y}\P\bra{Y=y}H\bra{X|Y=y},\]
and we have the identity
\[H\bra{X,Y}=H(Y) + H(X|Y) = H(X) + H(Y|X).\]
Entropy is always non-negative,
\[H(X) \geq 0,\]
and conditioning does not increase entropy,
\[H\bra{X|Y} \leq H\bra{X}.\]
Writing
\[\supp X = \{x : \P\bra{X=x} > 0\}\]
for the support of $X$, we have
\[H\bra{X} \leq \log \abs{\supp X}, \qquad \max_x \P\bra{X=x} \geq e^{-H(X)}.\]

For two random variables $X$ and $Y$, each supported on a finite set of integers, we define their entropic Ruzsa distance by
\[d\sqbra{X,Y}=H\bra{X'-Y'}-\frac{1}{2}H(X')-\frac{1}{2}H(Y'),\]
where $X'$ and $Y'$ are independent copies of $X$ and $Y$. Entropic Ruzsa distance is symmetric and invariant under translation in either distribution. If $X$ and $Y$ are independent, then
\[H\bra{X-Y} \geq \max\{H(X),H(Y)\},\]
and it follows that
\[d[X,Y] \geq 0, \qquad d[X,Y] \geq \frac{1}{2}\abs{H(X)-H(Y)}.\]

\section{Structure of Ahlfors--David regular sets}
For the rest of the paper, $A \sub [N]$ is a finite, $\bra{\al,M}$-Ahlfors--David regular set with $0 < \al < 1$ and $M \geq 2$.
\smallbreak
For a positive integer $R$, we define $\mathcal{I}_R$ as the following family of intervals of length $R$, which partition $\R$,
\[\mathcal{I}_R=\{[nR,\bra{n+1}R) : n \in \Z\}.\]
For any $x \in \R$, we write $I_R\bra{x}$ for the unique interval in $\mathcal{I}_R$ containing $x$.
\smallbreak
If $R=qr$ for positive integers $q$ and $r$, then each interval $I \in \mathcal{I}_R$ can be partitioned into $q$ subintervals of length $r$, which we denote by
\[I_k = [nR+kr, nR + \bra{k+1}r), \qquad 0 \leq k < q.\]
If $A \cap I \neq \emptyset$, we define $D_I$ to be a random variable taking values in $\{0,\ldots,q-1\}$ with distribution
\[\P\bra{D_I=k}=\frac{\abs{A \cap I_k}}{\abs{A\cap I}}.\]
Thus $D_I$ describes the distribution of the points of $A$ among the subintervals $I_k$ of $I$.
\smallbreak

We first prove that most points of $A$ lie in intervals such that $\abs{A \cap I} \geq \frac{R^{\al}}{40M}$, that for these intervals no subinterval contains too many elements of $A$, and that there are roughly $q^{\al}$ non-empty subintervals.
\begin{lemma}\thlabel{good}
Let $0 < \al < 1$. Then for every $2 \leq R \leq N$, if $X$ is uniform on an $\bra{\al,M}$-Ahlfors--David regular set $A \sub [N]$, then
\[\P\bra{\abs{A \cap I_R\bra{X}} < \frac{R^{\al}}{40M}} \leq \frac{1}{10}.\]

Moreover, if $R = qr \leq N$ and $I \in \mathcal{I}_R$ is such that
\[\abs{A \cap I} \geq \frac{R^{\al}}{40M},\]
then
\[\max_k \P \bra{D_I = k} \leq \frac{40M^2}{q^{\al}}\]
and
\[\frac{q^{\al}}{40M^2} \leq \abs{\supp D_I} \leq 8 M^2 q^{\al}.\]
\end{lemma}
The choice of the constants $\frac{1}{40}$ and $\frac{9}{10}$ in \thref{good} is fairly arbitrary, and has no special significance, except that they are sufficient for our purposes.

\begin{proof}
    Let $\mathcal{E}=\{I=[nR,\bra{n+1}R) \in \mathcal{I}_R : A \cap I \neq \emptyset, n \ \text{even}\}$, and define $\mathcal{O}$ similarly for the odd intervals. Without loss of generality, assume that $\abs{\mathcal{E}} \geq \abs{\mathcal{O}}$. For each interval $I \in \mathcal{E}$, choose one point $a_I \in A \cap I$. Then the intervals
    \[\sqbra{a_I - \frac{R}{2},a_I + \frac{R}{2}}\]
    are pairwise disjoint, and we have 
    \[\abs{A \cap\sqbra{a_I - \frac{R}{2},a_I + \frac{R}{2}}} \geq \frac{1}{M}\bra{\frac{R}{2}}^{\al} \geq \frac{R^{\al}}{2M},\] 
    by Ahlfors--David regularity. Therefore $\EE$ contains at most $2M \abs{A}R^{-\al}$ non-empty intervals, and hence
    \[\abs{\{I \in \mathcal{I}_R : A \cap I \neq \emptyset\}} \leq 4 M \abs{A}R^{-\al}.\]

   Therefore the number of points which lie in some $I \in \mathcal{I}_R$ with
    \[\abs{A \cap I} < \frac{R^{\al}}{40M}\]
    is at most
    \[\bra{4 M \abs{A}R^{-\al}} \cdot \bra{\frac{R^{\al}}{40M}}=\frac{1}{10} \abs{A}.\]
    Since $X$ is uniform on $A$, the first assertion follows.
    \smallbreak
    Now suppose that $R = qr \leq N$ and
    \[\abs{A \cap I} \geq \frac{R^{\al}}{40M},\]
    and that $0 \leq k < q$ is such that $A \cap I_k \neq \emptyset$. Choosing some $x \in A \cap I_k$, we have $I_k \sub [x-r,x+r]$, and hence by Ahlfors--David regularity
    \[\abs{A \cap I_k} \leq Mr^{\al}.\]
    Hence, for any $0 \leq k < q$,
    \[\P\bra{D_I = k} = \frac{\abs{A \cap I_k}}{\abs{A \cap I}} \leq \frac{Mr^{\al}}{\frac{R^{\al}}{40M}}=\frac{40M^2}{q^{\al}},\]
    and so in particular,
    \[\max_k \P \bra{D_I = k} \leq \frac{40M^2}{q^{\al}}.\]
    For the final part, since
    \[\sum_{0 \leq k < q}\P\bra{D_I = k} = 1,\]
    we have the lower bound
    \[\abs{\supp D_I} \geq \bra{\max_{k}\P\bra{D_I = k}}^{-1} \geq  \frac{q^{\al}}{40M^2}.\]
    It remains to prove the corresponding upper bound. Assume without loss of generality that at least half of the non-empty subintervals $I_k$ have $k$ even. If $r \geq 2$, for each such $k$ choose one point $x_k \in A \cap I_k$ and consider the intervals
    \[\sqbra{x_k - \frac{r}{2}, x_k + \frac{r}{2}}.\]
     These intervals are pairwise disjoint, and we have
     \[\abs{A \cap \sqbra{x_k - \frac{r}{2}, x_k + \frac{r}{2}}} \geq \frac{1}{M}\bra{\frac{r}{2}}^{\al} \geq \frac{r^{\al}}{2M}.\]
     All of these intervals lie in $[nR-r/2,(n+1)R+r/2]$. Choosing $x_0 \in A \cap I$, if $2R \leq N$ then
     \[[nR-r/2,(n+1)R+r/2] \sub [x_0-2R,x_0+2R],\]
     and so 
     \[\abs{A \cap[nR-r/2,(n+1)R+r/2]} \leq \abs{A \cap [x_0-2R,x_0+2R]} \leq M(2R)^{\al}.\] If $2R > N$, then $\abs{A} \leq MN^{\al} \leq M \bra{2R}^\al$. Summing the disjoint lower bounds therefore yields
    \[\frac{\abs{\supp D_I}}{2}\frac{r^{\al}}{2M} \leq 2MR^{\al},\]
    and hence
    \[\abs{\supp D_I} \leq 8M^2 q^{\al}.\]
    For $r=1$, each half-open interval $[k,k+1)$ contains at most one integer, and
    \[\abs{\supp D_I} = \abs{A \cap I} \leq MR^{\al} = Mq^{\al} \leq 8M^2 q^{\al}.\]
\end{proof}

The following lemma is a discrete version of the arithmetic progression avoidance lemma of Dyatlov and Zahl \cite[Proposition 6.13]{dyatlov2016spectral}.

\begin{lemma}\thlabel{AP} Let $P \sub \Z$ be a non-empty finite arithmetic progression, and let
\[L_r = \{\ell \in \Z: A \cap [\ell r,\bra{\ell+1}r) \neq \emptyset\}\]
be the indices of the intervals in $\mathcal{I}_r$ which contain a point of $A$. Then
\[\abs{P \cap L_r} \leq 6 M^2 \abs{P}^{\al}.\]
Consequently, if
\[\abs{P \cap L_r} \geq \de \abs{P}\]
for some $\de > 0$, then
\[\abs{P} \leq \bra{\frac{6M^2}{\de}}^{\frac{1}{1-\al}}.\]
\end{lemma}

\begin{proof}
If $\abs{P \cap L_r} \leq 2$, the claim is immediate, so assume that $\abs{P \cap L_r} \geq 3$, so in particular $\abs{P} \geq 3$. Let $t>0$ be the common difference of $P$. If $tr>N$, then $P$ contains at most one index from $L_r$, since all occupied indices lie in an interval of length at most $N/r$. We may therefore suppose that $tr \leq N$.
\smallbreak
For each $\ell \in P \cap L_r$, choose
\[a_{\ell} \in A \cap [\ell r, \bra{\ell + 1}r),\]
and let
\[B_{\ell} = [a_{\ell}-tr, a_{\ell}+tr].\]
By Ahlfors--David regularity,
\[\abs{A \cap B_{\ell}} \geq \frac{\bra{tr}^{\al}}{M}.\]

If $x \in B_{\ell}$, then $a_{\ell} \in [x-tr,x+tr]$, so $\ell$ lies in an interval of length $2t+1$. The indices $\ell$ belong to an arithmetic progression of common difference $t$, so there are at most three possibilities, hence every point of $\R$ belongs to at most three intervals $B_{\ell}$.
\smallbreak
Let $J$ be the smallest interval containing all of the $B_{\ell}$. The largest and smallest indices in $P \cap L_r$ differ by at most $\bra{\abs{P}-1}t$, and therefore
\[\text{length}\bra{J} < \bra{\abs{P}+2}tr \leq 2 \abs{P}tr.\]
If $2\abs{P}tr \leq N$, choose a point $a$ of $A \cap J$ and consider the interval $\sqbra{a-2\abs{P}tr,a+2\abs{P}tr}$, which contains $A \cap J$. By Ahlfors--David regularity,
\[\abs{\sqbra{a-2\abs{P}tr,a+2\abs{P}tr} \cap A} \leq M \bra{2\abs{P}tr}^{\al},\]
and hence
\[\abs{A \cap J} \leq M \bra{2\abs{P}tr}^{\al} \leq 2M \bra{\abs{P}tr}^{\al}.\]
If $2 \abs{P}tr > N$, then $\abs{A} \leq MN^{\al}$, and we have the same bound.
\smallbreak
Summing the lower bounds for the $B_{\ell}$ and using the fact that each point of $\R$ is only contained in at most three intervals $B_{\ell}$, we have
\[\abs{P \cap L_r}\frac{\bra{tr}^{\al}}{M} \leq \sum_{\ell \in P\cap L_r}\abs{A \cap B_{\ell}} \leq 3 \abs{A \cap J} \leq 6M \bra{\abs{P}tr}^{\al},\]
and hence
\[\abs{P \cap L_r} \leq 6M^2 \abs{P}^{\al}.\]

The final part follows by combining this inequality with $\abs{P \cap L_r} \geq \de \abs{P}$ and rearranging.
    
\end{proof}

\section{Local entropy growth}
In this section, we prove the local entropy-growth estimate used in the main argument. The key point is that small entropic Ruzsa distance, combined with the inverse theorem of Green, Manners and Tao and Freiman's theorem, forces a positive proportion of one of the local distributions to lie in a generalised arithmetic progression, which is incompatible with the sparsity on arithmetic progressions we proved in \thref{AP}.
\subsection{Setup}
We will use the following special case of \cite[Proposition 1.2]{green2025sumsets}, which gives a quantitative correspondence between entropic Ruzsa distance and doubling. In fact, it is enough for our purposes that the bounds in \thref{GMT} depend only on $\be$, and their precise quantitative form is not needed. A sufficient qualitative statement follows from earlier work of Tao \cite[Proposition 5.2]{tao2010sumset}.

\begin{prop}[Green-Manners-Tao]\thlabel{GMT}
If $X$ and $Y$ are finitely supported integer-valued random variables such that $d[X,Y] \leq \beta$, then there is a non-empty finite set $S \sub \Z$ such that
\[\abs{S-S} \leq 4 e^{12 \beta}\abs{S},\]
\[d[U_S,Y] \leq 6\beta + \log 2,\]
where $U_S$ is the discrete uniform distribution on $S$.
\end{prop}

Before stating Freiman's theorem, we need the following definition.
\begin{defn}
A \emph{generalised arithmetic progression} of rank $d$ is a set of the form
\[\mathcal{P}=\{x_0+n_1v_1+\cdots+n_dv_d:0\leq n_i<N_i,\ n_i \in \Z\},\]
where $x_0,v_1,\ldots,v_d\in\Z$ and $N_1,\ldots,N_d$ are positive integers. We say that $\mathcal{P}$ is \emph{proper} if every choice of
\[(n_1,\ldots,n_d)\in [0,N_1)\times\cdots\times[0,N_d)\]
gives a distinct element of $\mathcal{P}$. In this case,
\[|\mathcal{P}|=N_1\cdots N_d.\] 
\end{defn}

We will use the following form of Freiman's theorem. For more background and a proof of Freiman's theorem see, for example, Chapter 5 of \cite{tao2006additive}.
\begin{prop}\thlabel{Freiman}
    For every $K \geq 1$, there are $d(K) \in \N$ and $F(K) > 0$ such that every non-empty finite set $S \sub \Z$ which satisfies
    \[\abs{S-S} \leq K \abs{S}\]
    is contained in a proper generalised arithmetic progression $\mathcal{P}$ with
    \[\text{rank}\ \mathcal{P} \leq d(K), \qquad \abs{\mathcal{P}} \leq F(K) \abs{S}.\]
\end{prop}

\subsection{Two key lemmas}

The first lemma says that, if $X$ and $Y$ are close in entropic Ruzsa distance, then a positive proportion of $Y$ lies inside a generalised arithmetic progression. We use this to prove a local entropy growth lemma.
\begin{lemma}\thlabel{GAP}
For any $\beta > 0$, there are constants $c_{\beta},C_{\beta} > 0$ and an integer $d_{\be} > 0$ such that, if $X$ and $Y$ are finitely supported integer-valued random variables with
\[d[X,Y] \leq \beta,\]
then there is a proper generalised arithmetic progression $\PP \sub \Z$ such that
\[\text{rank} \ \PP \leq d_{\beta}, \qquad \abs{\PP} \leq C_{\beta} e^{H(Y)}, \quad \text{and} \quad\P \bra{Y \in \PP} \geq c_{\beta}.\]
    
\end{lemma}
\begin{proof}
    By \thref{GMT}, there is a finite set $S \sub \Z$ such that
    \[d[U_S,Y] \leq 6 \be + \log 2,\]
    where $U_S$ is the uniform distribution on $S$. Thus,
    \[\abs{\log \abs{S}-H(Y)} \leq 12 \be + 2 \log 2.\]
    Let $U_S'$ and $Y'$ be independent copies of $U_S$ and $Y$ respectively. Then
    \[H\bra{U_S'-Y'}\leq 6 \be + \log 2 + \frac{1}{2}\log\abs{S}+\frac{1}{2}H(Y') \leq \log \abs{S} + 12 \be + 2 \log 2.\]
There must be some $n \in \Z$ such that
\[\P\bra{U_S'-Y'=n} \geq e^{-12\be - 2\log 2}\abs{S}^{-1}.\]
Since $U_S'$ is uniform on $S$,
\[\P\bra{U_S'-Y'=n}=\abs{S}^{-1}\P\bra{Y' \in S-n},\]
and hence
\[\P\bra{Y \in S-n}=\P\bra{Y' \in S-n} \geq e^{-12 \be - 2\log 2}.\]
Let $T=S-n$. Then $T-T=S-S$, and hence by \thref{GMT},
\[\abs{T-T} \leq 4 e^{12\be}\abs{T}.\]
Thus by \thref{Freiman} with $K=4 e^{12\be}$, $T$ is contained in a proper generalised arithmetic progression $\PP$ with
\[\text{rank}\ \PP \leq d_{\be},\qquad \abs{\PP} \leq F_{\be}\abs{T},\]
where
\[F_{\be}=F\bra{4e^{12 \be}},\qquad d_{\be}=d\bra{4e^{12 \be}}.\]
Since
\[\abs{\log \abs{S}-H(Y)} \leq 12 \be + 2 \log 2,\]
it follows that
\[\abs{T}=\abs{S} \leq e^{12 \be + 2 \log 2}e^{H(Y)}.\]
The lemma follows after taking
\[C_{\be} = F_{\be}e^{12 \be + 2 \log 2}, \qquad c_{\be}=e^{-12 \be - 2 \log 2}.\]
    
\end{proof}

The second lemma shows we have local entropy growth on intervals where $\abs{A \cap I} \geq \frac{R^{\al}}{40M}$.
\begin{lemma}\thlabel{local}
For every $\be>0$, there is a constant $K_{\be}\geq1$, depending only on $\be$, such that the following holds for every $0<\al<1$. Let
    \[C = \ceil{\frac{K_\be}{\al \bra{1-\al}}}, \qquad q \geq M^C, \qquad R = qr \leq N,\]
    and let $I,J \in \mathcal{I}_R$ satisfy
    \[\min\{\abs{A \cap I},\abs{A \cap J}\} \geq \frac{R^{\al}}{40M}.\]
    Then
    \[d\sqbra{D_I,-D_J} > \be.\]
\end{lemma}
\begin{proof}
    Let $I=[nR,\bra{n+1}R)$ and $J = [mR,\bra{m+1}R)$. Define the translated variables
    \[\h{D_I}=D_I + nq, \qquad \h{D_J} = D_J + mq.\]
    Then
    \[\supp{\h{D_I}}, \supp{\h{D_J}} \sub L_r,\]
    and since entropic Ruzsa distance is invariant under translating either variable and under simultaneously negating both variables,
    \[d[D_I,-D_J]=d[-D_I,D_J]=d[-\h{D_I},\h{D_J}].\]
    Suppose for a contradiction that
    \[d\sqbra{D_I,-D_J} \leq \be.\]
    By \thref{GAP}, applied to $-\h{D_I}$ and $\h{D_J}$, there is a generalised arithmetic progression $\PP$ such that
    \[\text{rank}\ \PP \leq d_{\be}, \qquad \P\bra{\h{D_J} \in \PP} \geq c_{\be}\]
    and
    \[\abs{\PP} \leq C_{\be}e^{H(\h{D_J})} \leq C_{\be}\abs{\supp \h{D_J}} \leq 8C_{\be}M^2 q^{\al},\]
    where the last inequality follows from \thref{good}.
    \smallbreak
    Let
    \[B = \PP \cap \supp \h{D_J}.\]
    By \thref{good}, and the fact that $\P\bra{\h{D_J} \in \PP} \geq c_{\be}$, we have
    \[\abs{B} \geq \frac{c_{\be}q^{\al}}{40M^2} .\]
    Combining $\abs{\PP} \leq 8C_{\be}M^2 q^{\al}$ with the previous inequality, we obtain
    \[\frac{\abs{B}}{\abs{\PP}} \geq \frac{c_{\be}}{320C_{\be}M^4}.\]
    Choosing $K_{\be}$ large enough ensures that $\abs{B} \geq 2$ and thus $\PP$ has positive rank, write
    \[\PP = \{x_0 + n_1v_1 + \cdots + n_dv_d : 0 \leq n_i < N_i\},\]
    and choose $i$ for which $N_i$ is largest. Since $\PP$ is proper, we have
    \[N_i \geq \abs{\PP}^{1/d} \geq \abs{B}^{1/d_{\be}}.\]
    Fixing all $n_j$ with $j \neq i$, and allowing $n_i$ to vary, gives $\frac{\abs{\PP}}{N_i}$ fibres which partition $\PP$, where each is an arithmetic progression of length $N_i$. There must be one fibre $Q$ which satisfies
    \[\abs{Q} \geq \abs{B}^{1/d_{\be}}, \qquad \frac{\abs{Q \cap B}}{\abs{Q}} \geq \frac{\abs{B}}{\abs{\PP}} \geq \frac{c_{\be}}{320C_{\be}M^4}.\]

    Since $B \sub \supp \h{D_{J}} \sub L_r$, \thref{AP} with $\de = \frac{c_{\be}}{320C_{\be}M^4}$ gives
    \[\abs{Q} \leq \bra{\frac{1920C_{\be}}{c_{\be}}}^{\frac{1}{1-\al}}M^{\frac{6}{1-\al}}.\]
    On the other hand, $q \geq M^C$ gives
    \[\abs{Q} \geq \abs{B}^{1/d_{\be}} \geq \bra{\frac{c_{\be}}{40}}^{1/d_{\be}}M^{\frac{\al C - 2}{d_{\be}}}.\]
    Thus combining these, we have
    \[\bra{\frac{c_{\be}}{40}}^{1/d_{\be}}M^{\frac{\al C - 2}{d_{\be}}} \leq \bra{\frac{1920C_{\be}}{c_{\be}}}^{\frac{1}{1-\al}}M^{\frac{6}{1-\al}}.\]
Raising this inequality to the power of $d_{\be}\bra{1-\al}$, and choosing $c_{\be} \leq 1, C_{\be} \geq 1$ we have
\[M^{\bra{\al C - 2}\bra{1-\al}-6d_\be} \leq \bra{\frac{40}{c_{\be}}}^{1-\al} \bra{\frac{1920C_{\be}}{c_{\be}}}^{d_{\be}} \leq {\frac{40}{c_{\be}}} \bra{\frac{1920C_{\be}}{c_{\be}}}^{d_{\be}}.\]
On the other hand, using the fact that
\[C = \ceil{\frac{K_{\be}}{\al \bra{1-\al}}},\]
we have
\[\bra{\al C - 2}\bra{1-\al}-6d_{\be} \geq K_{\be} - 2 - 6d_{\be}.\]
Since $M \geq 2$, we can thus obtain a contradiction by choosing $K_{\be}$ suitably large so that
\[K_{\be} > 2 + 6 d_{\be} + \frac{\log \bra{{\frac{40}{c_{\be}}} \bra{\frac{1920C_{\be}}{c_{\be}}}^{d_{\be}}}}{\log 2}.\]
\end{proof}

\section{Accumulating entropy growth across scales}
In this section, we prove \thref{main} by accumulating the local entropy gains from the previous section across successive scales.

\subsection{Entropy across a block boundary}
The following lemma gives the required entropy decomposition across a block boundary, with a loss of at most $\log 2$ arising from the possible carry.
\begin{lemma}\thlabel{block}
    Let $B \geq 2$ be an integer, and let $X,Y$ be finitely supported independent integer-valued random variables. Write
    \[X = BU + \widetilde{X}, \qquad Y = BV + \widetilde{Y}, \qquad 0 \leq \widetilde{X},\widetilde{Y} < B.\]
    Then
    \[d\sqbra{X,-Y} \geq d\sqbra{U,-V} + \E_{u,v} d\sqbra{\widetilde{X} | U=u, -\bra{\widetilde{Y}|V=v}}-\log 2,\]
    where the expectation is with respect to the joint laws of the conditioning variables.
\end{lemma}
\begin{proof}
Let
\[C = \floor{\frac{\widetilde{X}+\widetilde{Y}}{B}} \in \{0,1\}\]
record whether or not there is a carry, so that
\[\floor{\frac{X+Y}{B}} = U+V+C.\]

Thus $X+Y$ determines $U+V$ up to at most two possibilities, and hence
\[H\bra{U+V|X+Y} \leq \log 2.\]
Applying the chain rule for conditional entropy to the joint entropy $H(X+Y,U+V)$ in two different ways and equating the resulting expressions, we have
\[H(X+Y) = H(U+V) + H (X+Y|U+V) - H(U+V| X+Y).\]
We also have
\[H(X+Y|U+V) \geq H\bra{X+Y|U,V} = H\bra{\widetilde{X}+\widetilde{Y}|U,V}.\]
Combining the previous three displayed equations, we obtain the inequality
\[H\bra{X+Y} \geq H(U+V) + H\bra{\widetilde{X}+\widetilde{Y}|U,V} - \log 2.\]
Since the quotient and remainder determine the original integer uniquely,
\[H(X) = H(U) + H(\widetilde{X}|U), \qquad H(Y) = H(V) + H(\widetilde{Y}|V).\]

Since $X$ and $Y$ are independent,
\[d[X,-Y] = H(X+Y) - \frac{1}{2}H(X) - \frac{1}{2}H(Y)\]
\[\geq H(U+V) + H\bra{\widetilde{X}+\widetilde{Y}|U,V} - \log 2 - \frac{1}{2}\bra{H(U) + H(\widetilde{X}|U)} - \frac{1}{2}\bra{H(V) + H(\widetilde{Y}|V)}\]
\[=d[U,-V] + H\bra{\widetilde{X}+\widetilde{Y}|U,V}-\frac{1}{2}\bra{H(\widetilde{X}|U)+H(\widetilde{Y}|V)} - \log 2.\]
But expanding the conditional entropies as expectations over $U=u$ and $V=v$, we have
\[H\bra{\widetilde{X}+\widetilde{Y}|U,V}-\frac{1}{2}\bra{H(\widetilde{X}|U)+H(\widetilde{Y}|V)}=\E_{u,v} d\sqbra{\widetilde{X} | U=u, -\bra{\widetilde{Y}|V=v}},\]
proving the result.
\end{proof}

\subsection{Completing the proof}

We can now combine everything to prove the main result.
\begin{proof}[Proof of \thref{main}.]
    If $\al=0$, then $\abs{A}\leq M$. If $\abs{A}=1$ the result is trivial, while if $\abs{A}\geq2$ then
    \[\log\frac{\abs{A+A}}{\abs{A}}\geq\log\frac32\geq\frac{\log(3/2)}{\log M}\log\abs{A}.\]
    Thus the result holds with $c=\log(3/2)$. We assume from now on that $0 < \al < 1$. Once again, $\abs{A}=1$ is trivial, so we assume $\abs{A} \geq 2$.
    \smallbreak
    Let $X$ and $Y$ be independent uniform random variables on $A$. Fix
    \[\be = \frac{200 \log 2}{81},\]
    write $K = K_{\be}$ and let
    \[C = \ceil{\frac{K}{\al\bra{1-\al}}}, \qquad q= \ceil{M^C}, \qquad m = \floor{\log_q N}.\]
    First suppose that $m \geq 1$, and define, for each $0 \leq j \leq m$, the random variables
    \[U_j = \floor{\frac{X}{q^j}}, \qquad V_j = \floor{\frac{Y}{q^j}}.\]
    For $0 \leq j < m$, let
    \[X_j =U_j - qU_{j+1},\qquad Y_j = V_j - q V_{j+1},\]
    so that $X_j$ and $Y_j$ are the $j$-th digit of the base $q$ expansion of $X$ and $Y$ respectively.
    \smallbreak
    Since $X$ and $Y$ are independent, for every $j$ so are $U_j$ and $V_j$, as functions of $X$ and $Y$. We apply \thref{block} with $B=q$ to $U_j$ and $V_j$, and for convenience we write
    \[\mathcal{E}_j = \E_{s,t}d\sqbra{X_j|U_{j+1}=s, -\bra{Y_j|V_{j+1}=t}},\]
    where the expectation is over the joint law of $U_{j+1}$ and $V_{j+1}$. Then \thref{block} gives
    \[d\sqbra{U_j, - V_j} \geq d\sqbra{U_{j+1}, -V_{j+1}} + \EE_j - \log 2.\]
    \smallbreak
    Fix $0 \leq j < m$, and put
    \[r = q^j, \qquad R = qr.\]
    For some value $s$ of $U_{j+1}$, define the intervals
    \[I(s) = [sR,\bra{s+1}R), \qquad I(s)_k = [sR + kr, sR + (k+1)r).\]
    Then the event $U_{j+1}=s$ is exactly the event $X \in I(s)$, and the event $U_{j+1}=s, X_j = k$ is exactly the event $X \in I(s)_k$, for every $0 \leq k < q$.
    \smallbreak
    Since $X$ is uniform on $A$, it follows that whenever $\P\bra{U_{j+1}=s} > 0$,
    \[\P \bra{X_j=k|U_{j+1}=s} = \frac{\abs{A \cap I(s)_k}}{\abs{A \cap I(s)}}=\P \bra{D_{I(s)}=k}.\]
    The same holds for $V_{j+1}=t$, with $J(t), J(t)_k$ defined analogously, giving the conditional law $D_{J(t)}$ for $Y_j$. These conditional variables $D_{I(s)}$ and $D_{J(t)}$ are independent, since the conditioning events involve $X$ and $Y$ separately. Therefore, if $I=I_R(X)$ and $J=I_R(Y)$ then
    \[\EE_j = \E_{I,J} d[D_I,-D_J] = \sum_{I,J} \frac{\abs{A \cap I}\abs{A \cap J}}{\abs{A}^2}d[D_I,-D_J],\]
    where the sum is over the intervals of $\mathcal{I}_R$ which contain a point in $A$.
    \smallbreak
    For each $R$, let $\mathcal{G}_R$ denote the good intervals in $\mathcal{I}_R$, meaning $I \in \mathcal{I}_R$ with
    \[\abs{A \cap I} \geq \frac{R^{\al}}{40M}.\]
    By \thref{good},
    \[\sum_{I \in \mathcal{G}_R}\frac{\abs{A \cap I}}{\abs{A}} = \P\bra{I_R(X) \in \mathcal{G}_R} \geq \frac{9}{10}.\]
    By \thref{local}, for every $I,J \in \mathcal{G}_R$, we have
    \[d[D_I,-D_J] > \be.\]
    Thus, we have
    \[\sum_{I,J} \frac{\abs{A \cap I}\abs{A \cap J}}{\abs{A}^2}d[D_I,-D_J] \geq \be \sum_{I,J \in \mathcal{G}_R}\frac{\abs{A \cap I}\abs{A \cap J}}{\abs{A}^2}\]
    \[ \geq \be \bra{\sum_{I \in \mathcal{G}_R}\frac{\abs{A \cap I}}{\abs{A}}}^2\]
    \[\geq \be \bra{\frac{9}{10}}^2 = 2 \log 2,\]
    and hence
    \[\EE_j \geq 2 \log 2.\]
    Combining this with $d\sqbra{U_j, - V_j} \geq d\sqbra{U_{j+1}, -V_{j+1}} + \EE_j - \log 2$, we have
    \[d[U_j,-V_j] \geq d[U_{j+1},-V_{j+1}]+ \log 2\]
    for every $0 \leq j < m$. Summing these inequalities and telescoping, and using the fact that $X=U_0$, $Y=V_0$, we have
    \[d[X,-Y] \geq d[U_m,-V_m] + m \log 2 \geq m \log 2.\]
    But since $X,Y$ have the same distribution, $d[X,-X]=d[X,-Y] = H(X+Y)-H(X)$, and since $H(X) = \log \abs{A}$, $H(X+Y) \leq \log \abs{A+A}$,
    \[d[X,-X] \leq \log \frac{\abs{A+A}}{\abs{A}}.\]
    Combining this with the above, we have
    \[m \log 2 \leq \log \frac{\abs{A+A}}{\abs{A}}.\]
    Since $m = \floor{\log_q N} \geq 1$,
    \[m \geq \frac{1}{2}\frac{\log N}{\log q},\]
    and consequently, since $q = \ceil{M^C}$ and hence $\log q \leq \bra{C+1}\log M$,
    \[\log \frac{\abs{A+A}}{\abs{A}} \geq \frac{\log 2}{2 (C+1)}\frac{\log N}{\log M} .\]
Since $m \geq 1$, we have $N \geq q \geq M^C$. Using the fact that $\abs{A} \leq MN^{\al}$, we therefore have
\[\log \abs{A} \leq \log M + \al \log N \leq \bra{\al + \frac{1}{C}}\log N.\]
Substituting this into the previous inequality gives
\[\log \frac{\abs{A+A}}{\abs{A}} \geq \frac{\log 2}{2\bra{C+1}\bra{\al + \frac{1}{C}}}\frac{\log \abs{A}}{\log M}.\]
Since
\[C \leq \frac{K}{\al \bra{1-\al}} + 1,\]
\[\bra{C+1}\bra{\al + \frac{1}{C}} = \al \bra{C+1} + 1 + \frac{1}{C} \leq \frac{K}{1-\al} + 2\bra{\al+1} \leq \frac{K+4}{1-\al}.\]
Consequently, we have
\[\log\frac{\abs{A+A}}{\abs{A}} \geq \frac{\bra{1-\al}\log 2 \log \abs{A}}{2 \bra{K+4}\log M}.\]
Exponentiating proves the result with
\[c = \frac{\log 2}{2\bra{K+4}}.\]
    \smallbreak
It remains to consider $m=0$. If $\abs A=1$, the result is immediate, so suppose that $\abs A\geq2$. Since $N<q$, Ahlfors--David regularity at radius $N$ gives
\[\abs A\leq MN^\al\leq Mq^\al.\]
Using $\log q\leq(C+1)\log M$, we obtain
\[\log\abs A\leq\bra{1+\al(C+1)}\log M.\]
By our choice of $C$,
\[1+\al(C+1)\leq1+\frac{K}{1-\al}+2\al \leq\frac{K+3}{1-\al}.\]
Consequently,
\[\frac{1-\al}{\log M}\log\abs A\leq K+3.\]
On the other hand, $\abs{A+A}\geq2\abs A-1$ and $\abs A\geq2$ imply
\[\log\frac{\abs{A+A}}{\abs A}\geq\log(3/2) \geq\frac{\log(3/2)}{K+3} \frac{1-\al}{\log M}\log\abs A.\]
Since
\[\frac{\log(3/2)}{K+3}\geq \frac{\log2}{2(K+4)},\]
exponentiating gives
\[\abs{A+A}\geq\abs A^{1+\frac{c(1-\al)}{\log M}}\]
with $c = \frac{\log2}{2(K+4)}$, completing the proof.
\end{proof}

\section{From discrete to continuous}
We now deduce the continuous formulation \thref{continuous} from the discrete theorem \thref{main}. Unlike the multiscale discretisation used by Dyatlov and Zahl \cite[Section 6.4]{dyatlov2016spectral}, only the scale $\de$ is needed here, since the multiscale analysis has already been carried out in the proof of \thref{main}.

\begin{proof}[Proof of \thref{continuous}.]
With $X$ as in \thref{continuous}, let 
\[A=\{k\in\Z:X\cap[k\de,(k+1)\de)\neq\emptyset\}.\]
Since $X\sub[0,1]$, \[A\sub\{0,\ldots,\lfloor\de^{-1}\rfloor\}.\]
Set 
\[N=\lfloor\de^{-1}\rfloor+1,\] 
so that $A+1\sub[N]$.
\smallbreak
For each $k\in A$, choose $x_k\in X\cap[k\de,(k+1)\de)$.  We first prove the upper Ahlfors--David bound. Fix $a\in A$ and $1\leq R\leq N$. For every $k\in A\cap[a-R,a+R]$, 
\[\mu(B(x_k,\de))\geq C^{-1}\de^\al.\]
If $y\in B(x_k,\de)$, then the interval $[k\de,(k+1)\de)$ meets $[y-\de,y+\de]$, and an interval of length $2\de$ meets at most three intervals of length $\de$. Thus every point belongs to at most three of these balls.
\smallbreak
Moreover, if $\abs{k-a}\leq R$, then 
\[\abs{x_k-x_a}\leq(R+1)\de,\] 
so 
\[B(x_k,\de)\sub B(x_a,(R+2)\de).\]
The upper Ahlfors--David bound also holds for radii larger than $1$, since $X\sub B(x_a,1)$ and $\mu$ is supported on $X$, and so
\[\mu(X)=\mu(B(x_a,1))\leq C,\]
and hence for $r \geq 1$,
\[\mu\bra{B\bra{x_a,r}} \leq C \leq Cr^{\al}.\]
Therefore,
\[\abs{A\cap[a-R,a+R]}C^{-1}\de^\al\leq3C(R+2)^\al\de^\al.\]
Since $R\geq1$ and $\al<1$, 
\[(R+2)^\al\leq3R^\al,\] 
and therefore 
\[\abs{A\cap[a-R,a+R]}\leq9C^2R^\al.\]
\smallbreak
For the reverse inequality, suppose first that $R\geq2$. Since $R\leq N$, 
\[\de\leq\frac{R\de}{2}\leq1.\]
Every point $x\in X\cap B(x_a,R\de/2)$ lies in some interval $[k\de,(k+1)\de)$ with $k\in A\cap[a-R,a+R]$. Indeed, 
\[\abs{k-a}\de\leq\abs{x-x_a}+\de\leq R\de.\]
Each interval $[k\de,(k+1)\de)$ has $\mu$-measure at most $C\de^\al$, and hence 
\[C^{-1}(R\de/2)^\al\leq C\de^\al\abs{A\cap[a-R,a+R]}.\]
It follows that 
\[\abs{A\cap[a-R,a+R]}\geq\frac{1}{2C^2}R^\al.\]
When $1\leq R<2$, the same lower bound follows from $a\in A\cap[a-R,a+R]$. Thus $A$ is $(\al,M)$-Ahlfors--David regular with $M=9C^2$.
\smallbreak
The intervals $[k\de,(k+1)\de)$, $k\in A$, cover $X$, so 
\[N_\de(X)\leq\abs A.\]
Conversely, any interval of length $\de$ meets at most three of these intervals, and therefore 
\[\frac{1}{3}\abs A\leq N_\de(X).\]
\smallbreak
Now let $s\in A+A$, say $s=k+\ell$. Then 
\[x_k+x_\ell\in \bra{X+X}\cap[s\de,(s+2)\de).\]
An interval of length $\de$ can meet $[s\de,(s+2)\de)$ for at most three values of $s$. Hence 
\[\abs{A+A}\leq3N_\de(X+X).\]
\smallbreak
Applying \thref{main} to the translate of $A$, 
\[N_\de(X+X)\geq\frac13\abs{A+A}\geq\frac13\abs A^{1+\frac{c\bra{1-\al}}{\log(9C^2)}}.\]
Since $C\geq2$, 
\[\log(9C^2)\leq6\log C.\]
After replacing $c$ by $c/6$ and using $\abs A\geq N_\de(X)$, we obtain 
\[N_\de(X+X) \geq \frac13N_\de(X)^{1+\frac{c\bra{1-\al}}{\log C}}.\] 
\end{proof}

\appendix
\section{The digit construction of Dyatlov--Zahl}
We describe explicitly the digit construction from \cite[Section 6.8.3]{dyatlov2016spectral}, which is stated as an obstruction for the additive energy, and show that it also gives the required sumset obstruction.
\smallbreak
Fix integers $b \geq 2$ and $k \geq 1$, let $N = b^{2k}$ and let
\[A = \left\{1 + \sum_{j=0}^{k-1} a_j b^{2j} : 0 \leq a_j < b\right\} \sub [N].\]
In other words, $A-1$ is the set of non-negative integers less than $N$ which have a zero in their odd digits in base $b$. We first show that $A$ is Ahlfors--David regular with dimension $\al = \frac{1}{2}$ and regularity constant $3b$.
\smallbreak
For each $0 \leq j \leq k$, we can partition $\R$ into intervals of the form
\[[1+nb^{2j}, 1 + \bra{n+1}b^{2j}), \qquad n \in \Z.\]
Each interval either contains no points of $A$, or exactly $b^j$ points of $A$ - the choice of interval fixes all but the final $j$ digits in base $b^2$, which can then be chosen freely. Fixing some $a \in A$, the interval $\sqbra{a-b^{2j},a+b^{2j}}$ contains the interval of the form $[1+nb^{2j}, 1 + \bra{n+1}b^{2j})$ which contains $a$, and hence contains at least $b^j$ elements of $A$. On the other hand, at most three intervals of the form $[1+nb^{2j}, 1 + \bra{n+1}b^{2j})$ can intersect with $\sqbra{a-b^{2j},a+b^{2j}}$, and so this interval contains at most $3b^j$ elements of $A$. Therefore,
\[b^j \leq \abs{A \cap \sqbra{a-b^{2j},a+b^{2j}}} \leq 3b^j.\]
Now let $1 \leq R \leq N$, and choose $0 \leq j < k$ such that
\[b^{2j} \leq R \leq b^{2j+2}.\]
By containment of the corresponding intervals,
\[b^j \leq \abs{A \cap \sqbra{a-R,a+R}} \leq 3b^{j+1}.\]
Since $b^j \leq R^{1/2} \leq b^{j+1}$, we have
\[\frac{\sqrt{R}}{b} \leq \abs{A \cap [a-R,a+R]} \leq 3b \sqrt{R}.\]
Hence $A$ is $\bra{\frac{1}{2},3b}$-Ahlfors--David regular, uniformly in $k$.
\smallbreak
We now compare the size of $A$ to the size of $A+A$. First, note that $\abs{A}=b^k$, since there are $b$ choices for each of the $k$ digits $a_0, \ldots a_{k-1}$. On the other hand,
\[A+A=\left\{2 + \sum_{j=0}^{k-1} a_j b^{2j} : 0 \leq a_j \leq 2b-2\right\},\]
and hence
\[\abs{A+A}=\bra{2b-1}^k \leq \bra{2b}^k = \abs{A}^{1+\frac{\log 2}{\log b}}.\]
Using $M = 3b$, we have
\[\frac{\log 2}{\log b} \leq \frac{\log 6}{\log M},\]
and therefore
\[\abs{A+A} \leq \abs{A}^{1+\frac{\log 6}{\log M}}.\]
\bigbreak
For other values of $0<\al < 1$, a similar construction works by ensuring that exactly $\ceil{j\al}$ of the lowest $j$ digits are free, and the rest are zero.

\printbibliography
\end{document}